\documentclass[11pt,a4paper]{article}

\usepackage{epsf,epsfig,amsfonts,amsgen,amsmath,amstext,amsbsy,amsopn,amsthm}
\usepackage{amsmath}
\usepackage{amsfonts,amsthm,amssymb,bm}
\usepackage{amsfonts}
\usepackage{graphics}
\usepackage{latexsym,bm}
\usepackage{amsfonts,amsthm,amssymb,bbding}
\usepackage{indentfirst}
\usepackage{graphicx}
\usepackage{color}

\usepackage[colorlinks=true,anchorcolor=blue,filecolor=blue,linkcolor=blue,urlcolor=blue,citecolor=blue]{hyperref}
\usepackage{float}
\usepackage{tikz,enumerate}
\usepackage[margin=2.5cm]{geometry}

\newtheorem{thm}{Theorem}[section]

\newtheorem{lem}[thm]{Lemma}

\newtheorem{claim}{Claim}[section]

\addtocounter{section}{0}

\begin{document}
\title{Spectral extremal results for triangle-free graphs with chromatic number at least four\footnote{Supported by the National Natural Science Foundation of China (Nos.\,12271162,\,12326372), and Natural Science Foundation of Shanghai (Nos. 22ZR1416300 and 23JC1401500) and The Program for Professor of Special Appointment (Eastern Scholar) at Shanghai Institutions of Higher Learning (No.\,TP2022031).}}
\author{{\bf Yinfen Zhu$^{a,b}$},
{\bf Huiqiu Lin$^{a}$}\thanks{Corresponding author: huiqiulin@126.com(H. Lin)} \\
\small $^{a}$ School of Mathematics, East China University of Science and Technology, \\
\small  Shanghai 200237, China\\
\small $^{b}$ School of Mathematics and Physics, Xinjiang Institute of Engineering, \\
\small   Urumqi, Xinjiang 830023, China\\
}

\date{\today}
\maketitle

{\flushleft\large\bf Abstract}
A graph is called $F$-free if it does not contain a copy of $F$.
Let $G(r,s)$ denote a $K_{r+1}$-free graph of order $n$ with chromatic number at least $s$
that maximizes the spectral radius.
Nikiforov [Linear Algebra Appl., 2007] proved the spectral Tur\'{a}n theorem,
which implies that $G(r,s)$ is the $r$-partite Tur\'{a}n graph $T_{n,r}$ for $s\leq r$.
Lin, Ning, and Wu [Combin. Probab. Comput., 2021] characterized the unique spectral extremal graph $G(2,3)$.
This result was later extended by Li and Peng [SIAM J. Discrete Math., 2023]  to all $s=r+1\geq 3$.
In this paper, we push the characterization further by determining the unique extremal graph $G(2,4)$ for all sufficiently large $n$. Specifically, we show that $G(2,4)$ is precisely a blow-up of the Gr\"{o}tzsch graph. Interestingly, under the same conditions, $G(2,4)$ also coincides with the unique edge-extremal graph identified by Ren, Wang, Wang, and Yang [arXiv:2404.07486v2].

\begin{flushleft}
\end{flushleft}

\section{Introduction}

Given a graph $H$, a graph is called \textit{$H$-free}
if it does not contain $H$ as a subgraph.
The \emph{Tur\'{a}n number} of $H$, denoted by $\mathrm{ex}(n,H)$,
is the maximum number of edges in an $n$-vertex $H$-free graph.
As one of the earliest results in extremal graph theory,
Tur\'{a}n's theorem \cite{Turan1941} states that
$T_{n,r}$ is the unique graph that attains the maximum number of edges over all $n$-vertex $K_{r+1}$-free graphs.
Here the Tur\'{a}n graph $T_{n,r}$ denotes the complete $n$-vertex $r$-partite graph with part sizes as
equal as possible.
Let $f(F,s)$ denote the maximum number of edges over all
$n$-vertex $F$-free graphs with chromatic number at least $s$.
Since the Tur\'{a}n graph $T_{n,r}$ is $K_{r+1}$-free and satisfies $\chi(T_{n,r})=4$,
Tur\'{a}n's theorem directly implies that $f(K_{r+1},s)=e(T_{n,r})$ for all $s\leq r$.
For the case $s=r+1=3$,
Erd\H{o}s showed that $f(K_3,3)=\lfloor\frac{(n-1)^2}{4}\rfloor+1$
(see \cite[p. 306]{Bondy2008}).
This was later generalized by Brouwer \cite{Brouwer1981},
who proved that  $f(K_{r+1},r+1)=e(T_{n,r})-\lfloor\frac{n}{r}\rfloor+1$ for $n\geq 2r+1$.
Beyond the case $s=r+1$, the behavior of $f(K_{r+1},s)$ for $s\geq r+2$ remains a challenging problem.
In particular, for $r=2$ and $s=4$, Ren, Wang, Wang, and Yang \cite{Ren2025+} determined $f(K_3,4)$
and showed that the extremal graphs are derived from the Gr\"{o}tzsch graph $F_1$ (see Figure \ref{fig-1.1A}).
As the smallest triangle-free 4-chromatic graph (see \cite{Chvatal1974}),
$F_1$ serves as the base for the extremal construction $F_1(n)$.
Here, the graph $F_1(n)$ is a graph obtained from $F_1$ by replacing $v_{13}$, $v_1$, $v_{23}$ and $y$ with independent sets $V_1$, $V_2$, $V_3$ and $W$, respectively, such that
$|V_1|+|V_2|+|V_3|=\left\lfloor\frac{n-7}{2}\right\rfloor$ (or $=\left\lceil\frac{n-7}{2}\right\rceil$),
$|W|=\left\lceil\frac{n-7}{2}\right\rceil$ (or $=\left\lfloor\frac{n-7}{2}\right\rfloor$),
and two vertices are adjacent if and only if their original vertices in $F_1$ are adjacent.

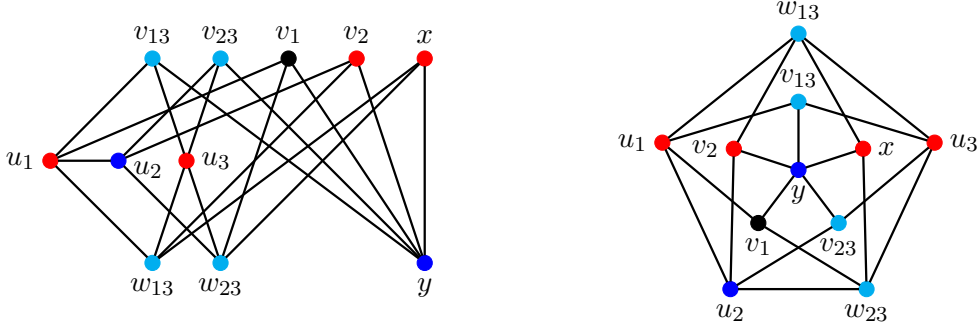
\begin{figure}[!h]
\centering
% 第一个子图：占页面45%宽度，左对齐
\begin{minipage}{0.45\textwidth}
\centering
\begin{tikzpicture}[scale=0.45, x=1.00mm, y=1.00mm, inner xsep=0pt, inner ysep=0pt, outer xsep=0pt, outer ysep=0pt]
\definecolor{L}{rgb}{0,0,0}
\definecolor{F}{rgb}{0,0,0}

\node[circle,fill=cyan,draw=cyan,inner sep=0pt,minimum size=2mm,label={[label distance=1mm]90:$v_{13}$}] (u1) at (0,30) {};
\node[circle,fill=cyan,draw=cyan,inner sep=0pt,minimum size=2mm,label={[label distance=1mm]90:$v_{23}$}] (u2) at (20,30) {};
\node[circle,fill=F,draw=F,inner sep=0pt,minimum size=2mm,label={[label distance=1mm]90:$v_1$}]
(u3) at (40,30) {};
\node[circle,fill=red,draw=red,inner sep=0pt,minimum size=2mm,label={[label distance=1mm]90:$v_2$}] (u4) at (60,30) {};
\node[circle,fill=red,draw=red,inner sep=0pt,minimum size=2mm,label={[label distance=1mm]90:$x$}]
(u5) at (80,30) {};

\node[circle,fill=red,draw=red,inner sep=0pt,minimum size=2mm,label={[label distance=1mm]180:$u_1$}] (u6) at (-30,0) {};
\node[circle,fill=blue,draw=blue,inner sep=0pt,minimum size=2mm,label={[label distance=1mm]357:$u_2$}] (u7) at (-10,0) {};
\node[circle,fill=red,draw=red,inner sep=0pt,minimum size=2mm,label={[label distance=1mm]0:$u_{3}$}] (u11) at (10,0) {};
\node[circle,fill=cyan,draw=cyan,inner sep=0pt,minimum size=2mm,label={[label distance=1mm]270:$w_{13}$}] (u8) at (0,-30) {};
\node[circle,fill=cyan,draw=cyan,inner sep=0pt,minimum size=2mm,label={[label distance=1mm]270:$w_{23}$}] (u9) at (20,-30) {};
\node[circle,fill=blue,draw=blue,inner sep=0pt,minimum size=2mm,label={[label distance=1mm]270:$y$}] (u10) at (80,-30) {};

\path[line width=0.3mm, draw=L] (u1) -- (u6);
\path[line width=0.3mm, draw=L] (u3) -- (u6);
\path[line width=0.3mm, draw=L] (u2) -- (u7);
\path[line width=0.3mm, draw=L] (u4) -- (u7);
\path[line width=0.3mm, draw=L] (u6) -- (u7);
\path[line width=0.3mm, draw=L] (u6) -- (u8);
\path[line width=0.3mm, draw=L] (u7) -- (u9);

\path[line width=0.3mm, draw=L] (u1) -- (u10);
\path[line width=0.3mm, draw=L] (u2) -- (u10);
\path[line width=0.3mm, draw=L] (u3) -- (u10);
\path[line width=0.3mm, draw=L] (u4) -- (u10);
\path[line width=0.3mm, draw=L] (u5) -- (u10);

\path[line width=0.3mm, draw=L] (u5) -- (u8);
\path[line width=0.3mm, draw=L] (u5) -- (u9);

\path[line width=0.3mm, draw=L] (u1) -- (u11);
\path[line width=0.3mm, draw=L] (u2) -- (u11);
\path[line width=0.3mm, draw=L] (u8) -- (u11);
\path[line width=0.3mm, draw=L] (u9) -- (u11);

\path[line width=0.3mm, draw=L] (u4) -- (u8);
\path[line width=0.3mm, draw=L] (u3) -- (u9);

%\draw(30,-40) node[anchor=base west]{\fontsize{14.23}{17.07}\selectfont $F_1$};

\end{tikzpicture}
\end{minipage}
\hspace{2mm}
% 第二个子图：占页面45%宽度，左对齐
\begin{minipage}{0.45\textwidth}
\centering
\begin{tikzpicture}[scale=0.45, x=1.00mm, y=1.00mm, inner xsep=0pt, inner ysep=0pt, outer xsep=0pt, outer ysep=0pt]
\definecolor{L}{rgb}{0,0,0}
\definecolor{F}{rgb}{0,0,0}

\node[circle,fill=blue,draw=blue,inner sep=0pt,minimum size=2mm,label={[label distance=1mm]270:$u_2$}] (rect1) at (-140,-35) {}; % 绘制一个长方形
\node[circle,fill=cyan,draw=cyan,inner sep=0pt,minimum size=2mm,label={[label distance=1mm]270:$w_{23}$}] (rect2) at (-100,-35) {}; % 绘制一个长方形

\node[circle,fill=blue,draw=blue,inner sep=0pt,minimum size=2mm,
label={[label distance=1mm]270:$y$}] (x) at (-120,0) {};

\node[circle,fill=cyan,draw=cyan,inner sep=0pt,minimum size=2mm,
label={[label distance=1mm]90:$v_{13}$}] (y) at (-120,20) {};

\node[circle,fill=cyan,draw=cyan,inner sep=0pt,minimum size=2mm,
label={[label distance=1mm]90:$w_{13}$}] (z) at (-120,40) {};

\node[circle,fill=red,draw=red,inner sep=0pt,minimum size=2mm,
label={[label distance=1mm]0:$x$}] (A1) at (-101,6.2) {};

\node[circle,fill=red,draw=red,inner sep=0pt,minimum size=2mm,
label={[label distance=1mm]180:$v_2$}] (B1) at (-139,6.2) {};

\node[circle,fill=cyan,draw=cyan,inner sep=0pt,minimum size=2mm,
label={[label distance=1mm]270:$v_{23}$}](Bx) at (-108.2,-15.6) {};

\node[circle,fill=L,draw=L,inner sep=0pt,minimum size=2mm,
label={[label distance=1mm]270:$v_1$}] (Ax) at (-131.8,-15.6) {};

\node[circle,fill=red,draw=red,inner sep=0pt,minimum size=2mm,
label={[label distance=1mm]0:$u_3$}] (A2) at (-80,8) {};

\node[circle,fill=red,draw=red,inner sep=0pt,minimum size=2mm,
label={[label distance=1mm]180:$u_1$}] (B2) at (-160,8) {};
%\draw(-28,7) node[anchor=base west]{\fontsize{10.23}{17.07}\selectfont $B_2$};

\definecolor{L}{rgb}{0,0,0}
\path[line width=0.3mm, draw=L] (x) -- (y);
\path[line width=0.3mm, draw=L] (x) -- (A1);
\path[line width=0.3mm, draw=L] (x) -- (B1);
\path[line width=0.3mm, draw=L] (x) -- (Ax);
\path[line width=0.3mm, draw=L] (x) -- (Bx);

\path[line width=0.3mm, draw=L] (y) -- (A2);
\path[line width=0.3mm, draw=L] (y) -- (B2);

\path[line width=0.3mm, draw=L] (z) -- (A2);
\path[line width=0.3mm, draw=L] (z) -- (B2);

\path[line width=0.3mm, draw=L] (B2) -- (Ax);
\path[line width=0.3mm, draw=L] (B2) -- (rect1);

\path[line width=0.3mm, draw=L] (A2) -- (Bx);
\path[line width=0.3mm, draw=L] (A2) -- (rect2);

\path[line width=0.3mm, draw=L] (Ax) -- (rect2);
\path[line width=0.3mm, draw=L] (Bx) -- (rect1);

\path[line width=0.3mm, draw=L] (rect1) -- (rect2);

\path[line width=0.3mm, draw=L] (A1) -- (rect2);
\path[line width=0.3mm, draw=L] (A1) -- (z);

\path[line width=0.3mm, draw=L] (B1) -- (rect1);
\path[line width=0.3mm, draw=L] (B1) -- (z);
\end{tikzpicture}
\end{minipage}

\caption{The graph $F_1$.}{\label{fig-1.1A}}

\end{figure}

\begin{thm} \emph{(\cite{Ren2025+})} \label{Theorem1.1}
Let $G$ be a graph on $n$ vertices with $n \geq 150$. If $G$ is triangle-free and $\chi(G) \geq 4$, then
    $$e(G) \leq \big\lfloor \frac{(n-3)^2}{4}\big\rfloor + 5,$$
with equality if and only if $G = F_1(n)$ up to isomorphism.
\end{thm}

Clearly, $K_3$ is an odd cycle.
For longer odd cycles,
a result of Ren, Wang, Wang, and Yang implies that
$f(C_{2\ell+1},3)=\lfloor\frac{(n-2)^2}{4}\rfloor+3$ for $\ell\geq 2$ and $n\geq 318\ell^2$; see \cite[Theorem 1.3]{Ren2024}.

The \emph{adjacency matrix} of \(G\), denoted by \(A(G)\),
is an \(n\times n\) matrix whose \((i,j)\)-entry is \(1\) if \(v_i\) is adjacent to \(v_j\) and \(0\) otherwise.
Let $\rho(G)$ denote the \emph{spectral radius} of the graph $G$.
In 2007, Nikiforov \cite{Nikiforov-2007} established a spectral version of Tur\'{a}n theorem,
namely that $T_{n,r}$ is the unique $n$-vertex $K_{r+1}$-free graph attaining the maximum spectral radius.
Subsequently, he \cite{Nikiforov2008} provided a
spectral condition for the existence of $C_{2\ell+1}$ in a graph.

\begin{thm} \emph{(\cite{Nikiforov2008})}\label{Theorem1.2}
Let $G$ be a graph of sufficiently large order $n$ with $\rho(G)\geq \rho(T_{n,2})$.
Then $G$ contains a copy of $C_{\ell}$ for every $\ell\leq n/{320}$.
\end{thm}

Theorem \ref{Theorem1.2} was first slightly refined by Ning and Peng \cite{Ning2020} for $n \geq 160t$.
Susequently, Zhai and Lin \cite{Zhai2023} employed different techniques to remove the ``sufficiently large order $n$'' condition, improving the bound to $n \geq 7t$. Building on this work, Li and Ning \cite{Li-N2023} further advanced the result, establishing the bound $n \geq 4t$.
A graph $H$ is called \emph{color-critical} if
 there exists an edge $e\in E(H)$ such that $\chi(H-\{e\})<\chi(H)$,
where $\chi(H)$ denotes the \emph{chromatic number} of $H$.
Let $SK_{a,b}$ denote the graph obtained from $K_{a,b}$ by subdividing an edge,
and $K_{a,b}\circ K_3$ denote the graph obtained by identifying a vertex of $K_{a,b}$
belonging to the part of size $b$ and a vertex of $K_3$.
Given that the unique extremal graph for Theorem \ref{Theorem1.2} is bipartite,
it is natural to consider this problem for non-bipartite graphs of order $n$.

\begin{thm}%\label{Theorem1.3}
Let $G$ be a $C_{2\ell+1}$-free graph of order $n$, where $\ell\geq 1$.

\vspace{1mm}
{\rm (i)} \emph{(\cite{Lin-N2021})}
If $\ell=1$ and $n\geq 5$, then $\rho(G)\leq \rho(SK_{\lceil\frac{n-1}{2}\rceil,\lfloor\frac{n-1}{2}\rfloor})$,
with equality if and only if $G\cong SK_{\lceil\frac{n-1}{2}\rceil,\lfloor\frac{n-1}{2}\rfloor}$.

\vspace{1mm}
{\rm (ii)} \emph{(\cite{Guo20212})}
If $\ell= 2$ and $n\geq 21$, then $\rho(G)\leq \rho(K_{\lceil\frac{n-2}{2}\rceil,\lfloor\frac{n-2}{2}\rfloor}\circ K_3)$,
with equality if and only if $G\cong K_{\lceil\frac{n-2}{2}\rceil,\lfloor\frac{n-2}{2}\rfloor}\circ K_3$.

{\rm (iii)} \emph{(\cite{Zhang2023})}
If $\ell\geq 2$ and $n$ is sufficiently large with respect to $k$, then $\rho(G)\leq \rho(K_{\lceil\frac{n-2}{2}\rceil,\lfloor\frac{n-2}{2}\rfloor}\circ K_3)$,
with equality if and only if $G\cong K_{\lceil\frac{n-2}{2}\rceil,\lfloor\frac{n-2}{2}\rfloor}\circ K_3$.
\end{thm}

Let $G(r,s)$ denote a $K_{r+1}$-free graph of order $n$ with chromatic number at least $s$
that maximizes the spectral radius.
Since the Tur\'{a}n graph $T_{n,r}$ is $K_{r+1}$-free and satisfies $\chi(T_{n,r})=4$,
the spectral Tur\'{a}n theorem directly implies that $G(r,s)=T_{n,r}$ for all $s\leq r$.
The problem becomes more interesting when $s\geq r+1$.
For the case $s=r+1=3$,
Lin, Ning, and Wu \cite{Lin-N2021} characterized the unique spectral extremal graph $G(2,3)$
for the case $s=r+1=3$.
This result was subsequently generalized to all $s=r+1\geq 3$ by Li and Peng \cite{Li-P2023}.
Despite these developments and other related works \cite{Liu2025,Nikiforov2010,Zou2025+},
determining $G(r,s)$ for $s\geq r+2$ remains an open and challenging problem.
Motivated by this gap,
we focus on the next natural case and characterize the unique extremal graph $G(2,4)$.
To state our findings, we first define a specific graph construction:
let $F_1(s,t)$ be the graph obtained from $F_1$ by replacing $x$ and $y$ with independent sets of sizes $s$ and $t$, respectively,
such that two vertices of $F_1(s,t)$ are adjacent if and only if their corresponding original vertices in $F_1$ are adjacent (see $F_1$ in Figure \ref{fig-1.1A}).

\begin{thm}\label{thm1.1}
For sufficiently large $n$, we have
 $G(2,4)=F_1\big(\big\lfloor\frac{n-11}{2}\big\rfloor,\big\lceil\frac{n-7}{2}\big\rceil\big)$.
\end{thm}

%\begin{thm}\label{thm1.1}
%Let $G$ be a triangle-free graph of order $n$ with $\chi(G)\geq 4$.
%For sufficiently large $n$, we have
%$$\rho(G)\leq \rho\Big(F_1\big(\big\lfloor\frac{n-11}{2}\big\rfloor,\big\lceil\frac{n-7}{2}\big\rceil\big)\Big),$$
%with equality if and only if $G\cong F_1\big(\big\lfloor\frac{n-11}{2}\big\rfloor,\big\lceil\frac{n-7}{2}\big\rceil\big)$.
%\end{thm}

Interestingly, by Theorem \ref{Theorem1.1}, the spectral extremal graph $F_1(\big\lfloor\frac{n-11}{2}\big\rfloor,\big\lceil\frac{n-7}{2}\big\rceil)$
is also an edge-extremal graph
with respect to $f(K_3,4)$ for all $n\geq 150$.

\section{Preliminaries}%\label{section2}

We begin by introducing some definitions and notations.
Let $G$ be a simple graph.
We denote its vertex set by $V(G)$, its edge set by $E(G)$, and the number of its edges by $e(G)$.
For two disjoint vertex subsets $U,V\subseteq V(G)$,
let $G[U]$ be the subgraph induced by $U$, and define $G-U=G[V(G)\setminus U]$.
Let $G[U,V]$ be the bipartite subgraph on the vertex set $U\cup V$
consisting of all edges with one endpoint in $U$ and the other in $V$.
When no confusion arises, we write $e(U)=e(G[U])$ and $e(U,V)=e(G[U,V])$, respectively.
Given a vertex $u\in V(G)$ and a subset $U\subseteq V(G)$ (possibly, $u\notin U$),
we denote by $N_G(u)$ the set of neighbors of $u$ in $G$,
and set $N_U(u)=N_G(u)\cap U$.
Let $d_G(u)=|N_G(u)|$ and $d_U(u)=|N_G(u)\cap U|$.
In this section, we first list some lemmas
that will be used in the proof of our main result.
The following is the spectral version of the Stability Lemma due to Nikiforov \cite{Nikiforov-2009}.

\begin{lem}\label{Lem2.1}\emph{(\cite{Nikiforov-2009})}
Let $r\ge 2$, $\frac{1}{\ln n}<c<r^{-8(r+21)(r+1)}$, $0<\varepsilon<2^{-36}r^{-24}$ and $G$ be an $n$-vertex graph.
If $\rho(G)>(1-\frac1r-\varepsilon)n$, then one of the following holds:

\vspace{1mm}
{\rm (i)} $G$ contains a $K_{r+1}(\lfloor c\ln n\rfloor, \dots,\lfloor c\ln n\rfloor,\lceil n^{1-\sqrt{c}}\rceil)$;

\vspace{1mm}
{\rm (ii)} $G$ differs from $T_{n,r}$ in fewer than $(\varepsilon^{\frac{1}{4}}+c^{\frac{1}{8r+8}})n^2$ edges.
\end{lem}

%From Lemma \ref{Lem2.1}, Desai et al. \cite{Desai-2022} derived the following stability result.
%Lemma \ref{Lem2.1} and the subsequent lemma provide an effective approach for studying spectral extremal problems.
%
%\begin{lem} \label{Lem2.2}\emph{(\cite{Desai-2022})}
%Let $F$ be a graph with  $\chi(F)=r+1$.
%For every $\varepsilon>0$, there exist $\delta>0$ and $n_0$ such that
%if $G$ is an $F$-free graph on $n\ge n_0$ vertices with $\rho(G)\geq (1-\frac1r-\delta)n$,
%then $G$ can be obtained from $T_{n,r}$ by adding and deleting at most $\varepsilon n^2$ edges.
%\end{lem}

%Based on Nikiforov's result \cite[Theorem 2]{Nikiforov-2009-2} and a more detailed analysis of the equality case in his proof, one can derive the following spectral version of the color-critical theorem, as presented by Zhai and Lin \cite[Theorem 1.2]{Zhai2023}.
%
%
%\begin{lem} \label{Lem2.3}\emph{(\cite{Nikiforov-2009-2, Zhai2023})}
%Let $r\geq 2$ and $H$ be a color-critical graph with $\chi(H)=r+1$.
%Then there exists an $n_0(H)\geq e^{|V(H)|r^{(2r+9)(r+1)}}$ such that
% ${\rm SPEX}(n,H)=\{T_{n,r}\}$ provided $n\geq n_0(H)$.
%\end{lem}

Let $K = K_r(n_1,\dots,n_r)$ be the complete $r$-partite graph
with partite classes of sizes $n_1\ge \cdots \ge n_r$.
An edge added inside a partite class is called a class-edge,
and an edge removed between two distinct partite classes is called a cross-edge.
The following result gives an estimate of \(\rho(G)\) for any graph $G$ that differs from a complete multipartite graph by only a small number of edges.

\begin{lem}[\cite{Fang2025+}]\label{Lem2.3}
Let $n$ be sufficiently large, and let $G$ be a graph obtained from an $n$-vertex complete $r$-partite graph $K=K_r(n_1,n_2,\dots,n_r)$ by
 adding $\alpha_1$ class-edges and deleting  $\alpha_2$ cross-edges,
where $\max\{\alpha_1,\alpha_2\} \le \frac{n}{(20r)^3}$.

\vspace{1mm}
{\rm (i)} If $n_1-n_r \le \frac{n}{400}$, then by denoting $\phi=\max\{n_1-n_r,2(\alpha_1 +\alpha_2)\}$, we have
$$\Big| \rho(G)\!-\!\rho(K)\!-\!\frac{2(\alpha_1\!-\!\alpha_2)}{n} \Big| \le\frac{56(\alpha_1\!+\!\alpha_2)\phi}{n^2}. $$

{\rm (ii)} If $n_1-n_r\geq 2k$ for some integer $k\le \frac{n}{(20r)^3}$, then we have
\begin{align*}
\rho(G) &\leq
    \rho(T_{n,r})\!+\!\frac{2(\alpha_1\!-\!\alpha_2)}{n}\!-\!\frac{2(r\!-\!1)k^2}{rn} \cdot \Big(1\!-\!\frac{28r\psi}{n} \Big)^4
    \!+\!\frac{56(\alpha_1\!+\!\alpha_2)\cdot 7r \psi}{n^2},
    \end{align*}
where $\psi=\max\{3k,2(\alpha_1+\alpha_2)\}$.
\end{lem}

\begin{lem}\label{Lem2.4}\emph{(\cite{Li-N2023})}
Let $u$ be a vertex of a graph $G$.
Then $\rho^2(G)\leq \rho^2(G-\{u\})+2d_G(u).$
\end{lem}

Next, we conduct a detailed comparison of the spectral radius for several specific graphs,\
which are crucial to the proof of our main result.

\begin{lem}\label{lem4.2}
Let $n$ be sufficiently large, $|s|\leq 22$,
and $G\cong F_1(\lfloor\frac{n-11-s}{2}\rfloor,\lceil\frac{n-7+s}{2}\rceil)$.
Then $$\rho(G)\leq \rho(F_1(\lfloor\frac{n-11}{2}\rfloor,\lceil\frac{n-7}{2}\rceil)),$$
with equality if and only if $G=F_1(\lfloor\frac{n-11}{2}\rfloor,\lceil\frac{n-7}{2}\rceil)$.
\end{lem}

\begin{proof}
Let $G^\star$ be a graph that maximizes the spectral radius over all graphs in
  $$\Big\{F_1\big(\lfloor\frac{n-11-s}{2}\rfloor,\lceil\frac{n-7+s}{2}\rceil\big)~:~|s|\leq 22\Big\}.$$
Thus, $G^\star=F_1(\frac{n-11-t}{2},\frac{n-7+t}{2})$ for some
integer $t$ with $|t|\leq 23$.
Set $\rho = \rho (G^\star)$.
The proof is completed by showing that $t=0$ for odd $n$ and $t=1$ for even $n$.
In either case, this yields $G^\star=F_1(\lfloor\frac{n-11}{2}\rfloor,\lceil\frac{n-7}{2}\rceil)$.

\begin{claim}\label{claim3.12A}
We have $\frac{n-3}{2}-\frac{13.2}{n}\leq \rho\big(F_1(\lfloor\frac{n-11}{2}\rfloor,\lceil\frac{n-7}{2}\rceil)\big)
\leq \rho\leq \frac{n}{2}$.
\end{claim}

\begin{proof}
Since $G^\star=F_1(\frac{n-11-t}{2},\frac{n-7+t}{2})$ is clearly an $n$-vertex triangle-free graph,
Theorem \ref{Theorem1.2} implies that $\rho\leq \rho(T_{n,2})\leq \frac{n}{2}$.

It remains to show the lower bound $\rho\geq \frac{n-3}{2}-\frac{13.2}{n}$.
To this end, let $F'=F_1(\lfloor\frac{n-11}{2}\rfloor,\lceil\frac{n-7}{2}\rceil)-\{u_1,u_2,u_3\}$ (see $F_1$ in Figure \ref{fig-1.1A}).
Then, $F'$ can be obtained from $T_{n-3,2}$ by deleting 6 edges.
Applying  Lemma \ref{Lem2.3} with  $\alpha_1=0$ and $\alpha_2=6$, we obtain
$| \rho(F')-\rho(T_{n-3,2})+\frac{12}{n}| \leq\frac{0.1}{n}$.
It follows that
\begin{align*}
\rho(F')\geq \rho(T_{n-3,2})-\frac{12.1}{n}\geq \sqrt{\big\lfloor\frac{(n-3)^2}{4}\big\rfloor}-\frac{12.1}{n}\geq \frac{n-3}{2}-\frac{13.2}{n}.
\end{align*}
Since $F'$ is a subgraph of $ G^\star$, we have $\rho\geq \rho\big(F_1(\lfloor\frac{n-11}{2}\rfloor,\lceil\frac{n-7}{2}\rceil)\big)\geq \frac{n-3}{2}-\frac{13.2}{n}$, as required.
\end{proof}

Let $\mathbf{x}=(x_1,\ldots,x_n)^{\mathrm{T}}$ be the positive unit eigenvector of $G^\star$.
Clearly, $G^\star$ is obtained from $F_1$ by replacing $x$ and $y$ with independent sets $A_x$ and $A_y$
of sizes $\frac{n-11-t}{2}$ and $\frac{n-7+t}{2}$, respectively,
such that two vertices of $F_1(\frac{n-11-t}{2},\frac{n-7+t}{2})$ are adjacent if and only if their corresponding original vertices in $F_1$ are adjacent.
We partition the vertex set of $G^\star$ as $\Pi$:
{\small \begin{align*}
V(G^\star)=\{v_{13}\}\cup \{v_{23}\} \cup \{v_{2}\}\cup A_x\cup \{u_{1}\}\cup \{u_{2}\}
  \cup \{u_{3}\}\cup \{w_{13}\}\cup \{w_{23}\}\cup \{w_{2}\}\cup A_y.
\end{align*}}
Thus, we see that $\rho$ is the largest eigenvalue of

\begingroup
\renewcommand{\arraystretch}{1.5} % 增加行高，防止分数重叠
\setlength{\arraycolsep}{7pt}    % 减小列间距
\footnotesize                    % 缩小字号
\[
B_{\Pi} = \begin{bmatrix}
0 & 0 & 0 & 0 & 0 & 1 & 0 & 1 & 0 & 0 & \frac{n-7+t}{2} \\
0 & 0 & 0 & 0 & 0 & 0 & 1 & 1 & 0 & 0 & \frac{n-7+t}{2} \\
0 & 0 & 0 & 0 & 0 & 1 & 0 & 0 & 0 & 1 & \frac{n-7+t}{2} \\
0 & 0 & 0 & 0 & 0 & 0 & 1 & 0 & 1 & 0 & \frac{n-7+t}{2} \\
0 & 0 & 0 & 0 & 0 & 0 & 0 & 0 & 1 & 1 & \frac{n-7+t}{2} \\
1 & 0 & 1 & 0 & 0 & 0 & 1 & 0 & 1 & 0 & 0 \\
0 & 1 & 0 & 1 & 0 & 1 & 0 & 0 & 0 & 1 & 0 \\
1 & 1 & 0 & 0 & 0 & 0 & 0 & 0 & 1 & 1 & 0 \\
0 & 0 & 0 & 1 & \frac{n-11-t}{2} & 1 & 0 & 1 & 0 & 0 & 0 \\
0 & 0 & 1 & 0 & \frac{n-11-t}{2} & 0 & 1 & 1 & 0 & 0 & 0 \\
1 & 1 & 1 & 1 & \frac{n-11-t}{2} & 0 & 0 & 0 & 0 & 0 & 0
\end{bmatrix}.
\]
\endgroup
That is, $\rho$ is the largest root of $\det(xI_{\Pi}-B_{\Pi})$.
Equivalently, $\rho-\frac{n}{2}$ is the largest root of $g(x,t):=\det\big((x+\frac{n}{2})I_{\Pi}-B_{\Pi}\big)$,
which yields that
\begin{align}\label{align-A001}
g(\rho-\frac{n}{2},t)=0.
\end{align}

Claim \ref{claim3.12A} gives that $\rho-\frac{n}{2}\in [-2,0]$.
To estimate the range of $\rho$, it is enough to consider the function $g(x,t)$ over the internal $x\in [-2,0]$.
Given $|t|\leq 23$, direct computation yields
\begin{align*}
g(x,t)&=\frac{(6+4x)n^{10}}{2048}+\frac{(39+t^2+108x+76x^2)n^9}{2048}\\
&+\frac{(-904+430x+864x^2+648x^3+18t^2x)n^8}{2048}\\
&+\frac{(2900-68t^2-112t-12400x+1536x^2+4032x^3+3264x^4+144t^2x^2)n^7}{2048}+o(n^7).
\end{align*}
We now divide the proof into two cases depending on the parity of $n$.

\noindent{\textbf{Case 1.} $n$ is odd.}

Let $\rho_1=\rho(F_1(\frac{n-11}{2},\frac{n-7}{2}))$.
Then $\rho_1-\frac{n}{2}$ is the largest root of
\begin{align*}
g_1(x):=g(x,0)&=\frac{(6+4x)n^{10}}{2048}+\frac{(39+108x+76x^2)n^9}{2048}
+o(n^9).
\end{align*}

Suppose that $t\neq 0$.
We can check that
$g(x,t)-g_1(x)=\frac{t^2}{2048}n^9+o(n^9)>0$ for every $x\in [-2,0]$.
Combined with \eqref{align-A001} and Claim \ref{claim3.12A},
$$g_1(\rho-\frac{n}{2})< g(\rho-\frac{n}{2},t) = 0=g_1(\rho_1-\frac{n}{2}).$$
Then $\rho_1>\rho$, which contradicts the choice of $G^\star$.
Therefore, $t=0$ and $G^\star=F_1(\frac{n-11}{2},\frac{n-7}{2})$.

\noindent{\textbf{Case 2.} $n$ is even.}

Let $\rho_2= \rho (F_1(\frac{n-12}{2},\frac{n-6}{2}))$.
Then $\rho_2-\frac{n}{2}$ is the largest root of
\begin{align*}
g_2(x) &:=g(x,1)=\frac{(6+4x)n^{10}}{2048}+\frac{(40+\!108x+76x^2)n^9}{2048}\\
&+\frac{(-904+448x+\!864x^2+648x^3)n^8}{2048}\\
&+\frac{(2900-68-112-12400x+1536x^2+4032x^3+3264x^4+144t^2x^2)n^7}{2048}+o(n^7).
\end{align*}

Suppose that $t\neq 1$.
For $x\in [-2,0]$,
if $t=-1$, then
$g(x,t)-g_2(x)\geq \frac{224}{2048}n^7+o(n^7)>0$;
if $t\notin \{-1,1\}$, then
$g(x,t)-g_2(x)\geq \frac{t^2-1}{2048}n^9+o(n^9)>0$.
Combining both cases with \eqref{align-A001} yields
$$g_2(\rho-\frac{n}{2})< g(\rho-\frac{n}{2},t) = 0=g_2(\rho_2-\frac{n}{2}).$$
Then $\rho_2>\rho$, which contradicts the choice of $G^\star$.
Therefore, $t=1$ and $G^\star=F_1(\frac{n-12}{2},\frac{n-6}{2})$.

This completes the proof of Lemma \ref{lem4.2}.
\end{proof}

\section{Proof of Theorem \ref{thm1.1}}%\label{section4}

In this section,
we are ready to give the proof of Theorem \ref{thm1.1}.

\begin{proof}[\textbf{Proof of Theorem \ref{thm1.1}}]
Let $n$ be sufficiently large, and $G$ be an $n$-vertex triangle-free graph with chromatic number at least four
that maximizes the spectral radius.
Clearly, $G$ is connected.
Otherwise, let $G_1$ and $G_2$ be two components of $G$ such that $\rho(G_1)=\rho(G)$.
Adding an edge between $G_1$ and $G_2$ yields a new triangle-free graph with strictly larger spectral radius,
which leads to a contradiction.
The proof proceeds via a series of claims presented below.

\begin{claim}\label{Claim3.1}
We have $\rho(G)\geq \frac{n-3}{2}-\frac{13.2}{n}$ and $e(G)\geq \frac{(n-4)^2}{4}$.
\end{claim}

\begin{proof}
Claim \ref{claim3.12A} gives that
$\rho\big(F_1(\lfloor\frac{n-11}{2}\rfloor,\lceil\frac{n-7}{2}\rceil)\big)\geq \frac{n-3}{2}-\frac{13.2}{n}$.
By the extremality of $G$, it follows that
$\rho(G)\geq \rho_1(F(\lfloor\frac{n-11}{2}\rfloor,\lceil\frac{n-7}{2}\rceil))$.
Thus, $\rho(G)\geq \frac{n-3}{2}-\frac{13.2}{n}$.

A result due to Nosal \cite{Nosal1970} states that $\rho(H)\leq \sqrt{e(H)}$ for every triangle-free graph $H$.
Applying this to $G$, we obtain
$e(G)\geq \rho^2(G)\geq \frac{(n-4)^2}{4}$.
\end{proof}

Let $\varepsilon$ be a positive constant, and let $\delta_{\ref{Lem2.1}}$ and $c_{\ref{Lem2.1}}$ be the constants provided by Lemma \ref{Lem2.1},
chosen such that
\begin{align}\label{Equ001}
\varepsilon\leq \varepsilon_{\ref{Lem2.1}}^{\frac{1}{4}}+c_{\ref{Lem2.1}}^{\frac{1}{8r+8}}\leq \frac{1}{10^4}.
\end{align}

\begin{claim}\label{Claim3.2}
Let $V_1\cup V_2$ be a vertex partition of $G$ that maximizes $e(V_1,V_2)$.
Then $e(V_1)+e(V_2)\leq \varepsilon n^2$ and $\big||V_i|-\frac{n}{2}\big|\leq 2\varepsilon^{\frac12} n$ for each $i\in \{1,2\}.$
\end{claim}

\begin{proof}
By Claim \ref{Claim3.1}, we have  $\rho(G)\geq \frac{n-4}{2}\geq (\frac12-\varepsilon)n$.
Since $G$ is triangle-free,
it follows from \eqref{Equ001} and
 Lemma \ref{Lem2.1} that $G$ can be obtained from $T_{n,2}$
by adding and deleting at most $\varepsilon n^2$ edges.
Then there exists a partition $V(G)=U_1\cup U_2$ such that
$\lfloor\frac{n}{2}\rfloor\le |U_1|\leq|U_2|\leq \lceil\frac{n}{2}\rceil$ and $e(U_1)+e(U_2)\leq \varepsilon n^2$.
Since $e(V_1,V_2)$ is maximum, it follows that $e(V_1)+e(V_2)$ is minimum, and hence
\begin{center}
$e(V_1)+e(V_2)\leq e(U_1)+e(U_2)\leq \varepsilon n^2.$
\end{center}
Set $\alpha=|V_1|-\frac{n}{2}$. Then
\begin{center}
  $e(G)= e(V_1,V_2)+e(V_1)+e(V_2)\leq \frac{n^2}{4}-\alpha^2+\varepsilon n^2.$
\end{center}
Combining with $e(G)\geq \frac{(n-4)^2}{4}$,
we get $\alpha^2\leq \varepsilon n^2+\frac{9n}{4}$,
and so $|\alpha|\leq 2\varepsilon^{\frac12} n$.
\end{proof}

\begin{claim}\label{Claim3.3}
Let $S=\{v\in V(G)~:~d_G(v)\leq \frac{2}{5}n\}$.
Then $|S|\leq 20$.
\end{claim}

\begin{proof}
By way of contradiction, assume that $|S|>20$.
Then there exists a subset $S'\subseteq S$ with $|S'|=20$.
Combining with $e(G)\geq \frac{(n-4)^2}{4}$ (see Claim \ref{Claim3.1}),  we get
\begin{equation}\label{Ali-001}
e(G-S')\geq e(G)-\sum_{v\in S'}d_G(v)
\geq \frac{(n-4)^2}{4}-20\cdot \frac{2}{5}n
>\frac{(n-20)^2}{4}.
\end{equation}

As $G$ is triangle-free,  the subgraph $G-S'$ must also be triangle-free.
Tur\'{a}n's theorem then implies $e(G-S')\leq e(T_{n-20,2})$,
which contradicts \eqref{Ali-001}.
We conclude that $|S|\leq 20$.
\end{proof}

\begin{claim}\label{claim3.4}
For every $i\in \{1,2\}$, we define $V_i'=V_i\setminus S$.
Then $G[V_i']$ is empty.
\end{claim}

\begin{proof}
By way of contradiction.
Without loss of generality, we may assume that $G[V_1']$ contains an edge $u_1u_2$.
Since $u_1\notin S$, we have $d_G(u_1)\geq \frac{2}{5}n$.
Moreover, since $V(G)=V_1\cup V_2$ is a partition such that $e(V_1,V_2)$ is maximum,
we have $d_{V_2}(u_1)\geq \frac{1}{2}d_G(u_1)\geq \frac{1}{5}n$.

Set $W_1=\{v\in V_1~|~d_{V_1}(v)\geq 2\varepsilon^{\frac12} n\}$.
Clearly,
\begin{center}
  $2e(V_1)=\sum\limits_{v\in V_1}d_{V_1}(v)\ge
\sum\limits_{v\in W_1}d_{V_1}(v)\ge |W_1|\cdot 2\varepsilon^{\frac12} n.$
\end{center}
Combining this with Claim \ref{Claim3.2} gives
  $\varepsilon n^2\geq e(V_1)\geq |W_1| \varepsilon^{\frac12} n,$
which yields $|W_1|\leq\varepsilon^{\frac{1}{2}} n$.

We first prove that $u_1\notin W_1$.
Suppose for contradiction that $u_1\in W_1$.
By the definition of $W_1$, we know that $d_{V_1}(u_1)\geq 2\varepsilon^{\frac12} n>\varepsilon^{\frac12} n+20$.
It then follows from Claim \ref{Claim3.3} that $d_{V_1}(u_1)>|W_1\cup S|$.
Hence, there exists a vertex $u_3\in N_{V_1}(u_1)\setminus (W_1\cup S)$.
By the definitions of $W_1$ and $S$,
we obtain that $d_{V_1}(u_3)< 2\varepsilon^{\frac12} n$ and $d_G(u_3)\geq \frac{2}{5}n$,
which implies that
\begin{align}\label{Ali-002}
d_{V_2}(u_3)\geq d_G(u_3)-d_{V_1}(u_3)\geq \big(\frac{2}{5}-2\varepsilon^{\frac12}\big)n.
\end{align}
Combining this with $d_{V_2}(u_1)\geq \frac{1}{5}n$ gives that
$$|N_{V_2}(u_1)\cap N_{V_2}(u_3)|\geq d_{V_2}(u_1)+d_{V_2}(u_3)-|V_2|\geq \big(\frac{1}{10}-4\varepsilon^{\frac12}\big)n,$$
where the last inequality follows from Claim \ref{Claim3.2} that $\big||V_2|-\frac{n}{2}\big|\leq 2\varepsilon^{\frac12} n$.
Since $\varepsilon\leq \frac{1}{10^4}$,
 there exists a vertex $u_4\in N_{V_2}(u_1)\cap N_{V_2}(u_3)$,
which implies that $\{u_1,u_3,u_4\}$ induces a triangle, a contradiction.
Thus, $u_1\notin W_1$.
Similarly, we also get  $u_2\notin W_1$.

Since $u_1,u_2\notin W_1$,
a similar argument as in \eqref{Ali-002} shows that $d_{V_2}(u_i)\geq \big(\frac{2}{5}-4\varepsilon^{\frac12}\big)n$
for each $i\in \{1,2\}$.
Consequently,
$$|N_{V_2}(u_1)\cap N_{V_2}(u_2)|\geq d_{V_2}(u_1)+d_{V_2}(u_2)-|V_2|\geq \big(\frac{3}{10}-10\varepsilon^{\frac12}\big)n,$$
where the last inequality follows from Claim \ref{Claim3.2}.
Thus, there exists a vertex $u_5\in N_{V_2}(u_1)\cap N_{V_2}(u_2)$,
which implies that $\{u_1,u_2,u_5\}$ induces a triangle, a contradiction.

This completes the proof of Claim \ref{claim3.4}.
\end{proof}

For any $i\in \{1,2\}$ and any $S'\subseteq S$, we define $V_i(S')=\{v\in V_i'~:~N_S(v)=S'\}$.
Note that $V_i(S')$ may be empty for some $S'\subseteq S$.

\begin{claim}\label{claim3.5}
Let $S_1,S_2\subseteq S$ (possibly, $S_1=S_2$) such that both $V_1(S_1)$ and $V_2(S_2)$ are non-empty.

\vspace{1mm}
{\rm (i)}  If $S_1\cap S_2\neq \varnothing$, then $G[V_1(S_1),V_2(S_2)]$ is an empty graph.

\vspace{1mm}
{\rm (ii)}  If $S_1\cap S_2=\varnothing$, then $G[V_1(S_1),V_2(S_2)]$ is a complete bipartite graph.
\end{claim}

\begin{proof}
(i) Suppose, for contradiction, that there exists an edge $u_1u_2\in E(V_1(S_1),V_2(S_2))$,
where $u_i\in V_i(S_i)$ for $i=1,2$.
Since $S_1\cap S_2\neq \varnothing$, there exists a vertex $u_3\in S_1\cap S_2$.
Then $\{u_1,u_2,u_3\}$ induces a triangle, a contradiction.

(ii) Suppose, for contradiction, that there exist a vertex $u_1\in V_1(S_1)$ and a vertex $u_2\in V_2(S_2)$
such that $u_1$ and $u_2$ are non-adjacent in $G$.
Let $G'$ be the graph obtained from $G$ by adding the edge $u_1u_2$.
We first show that $G'$ is triangle-free.
Indeed, if not, then there exists a triangle contains the edge $u_1u_2$.
The remaining vertex in this triangle is denoted by $u_3$.
By Claim \ref{claim3.4}, $G[V_i']$ ($i=1,2$) is empty, which implies that $u_3\in S$.
Since $u_1u_3,u_2u_3\in E(G)$, it follows that $u_3\in S_1$ and $u_3\in S_2$,
which contradicts $S_1\cap S_2=\varnothing$.
Thus, $G'$ is triangle-free.
Nevertheless, $\rho(G')>\rho(G)$ and $\chi(G')\geq \chi(G)\geq 4$,
which contradicts the extremal choice of $G$.
\end{proof}

Recall that $G$ is connected.
Then by the Perron-Frobenius theorem,
there is a positive unit eigenvector
$\mathbf{x}=(x_1,\ldots,x_n)^{\mathrm{T}}$ corresponding to $\rho(G)$,
where $x_{u^*}=\max\{x_v: v\in V(G)\}$.

\begin{claim}\label{claim3.6}
We have $|V_1(\varnothing)|\geq (\frac{1}{2}-2\varepsilon^{\frac12})n$
and $|V_1(S')|\leq 1$ for any non-empty $S'\subseteq S$.
\end{claim}

\begin{proof}
We first prove that there exists at most one subset $S'$ of $S$ such that $|V_1(S')|\geq 2$.
Otherwise, there exist two distinct subsets $S_1,S_2$ of $S$ such that $|V_1(S_1)|\geq 2$ and $|V_1(S_2)|\geq 2$.
By definition, we know that $V_1(S_1) \cap V_1(S_2)=\varnothing$.
Let $u_1,u_2\in V_1(S_1)$ and $v_1,v_2\in V_1(S_2)$ be arbitrary distinct vertices.
From Claim \ref{claim3.5},
it follows that $N_{V_2'}(u_1)=N_{V_2'}(u_2)$.
By Claim \ref{claim3.4}, we further obtain
$$N_{G}(u_1)=S_1\cup N_{V_2'}(u_1)=S_1\cup N_{V_2'}(u_2)=N_{G}(u_2),$$
and $x_{u_1}=x_{u_2}$.
Similarly, we also get $N_{G}(v_1)=N_{G}(v_2)$ and $x_{v_1}=x_{v_2}$.
Define
 $$\mathcal{A}_i:=\{S'~:~S'\subseteq S,~V_i(S')\neq \varnothing\}$$
 for $i=1,2$.
By Claim \ref{Claim3.3}, we have $|S|\leq 20$,
which implies that $|\mathcal{A}_i|\leq 2^{20}$.
Then by Claim \ref{claim3.5},
 $G$ is a blow-up of some graph $F$ with $|V(F)|\leq |S|+\sum_{i=1}^2|\mathcal{A}_i|\leq 2^{22}$.
It is not hard to verify that $\chi(F)=\chi(G)\geq 4$.

Assume without loss of generality that $x_{u_1}\geq x_{v_1}$.
Clearly, $\rho(G)x_{u_1}=\sum_{v\in N_{G}(u_1)}x_v$ and $\rho(G)x_{v_1}=\sum_{v\in N_{G}(v_1)}x_v$.
Define $G'$ to be the graph obtained from $G$ by deleting all edges incident to $v_1$ and
joining $v_1$ to $N_G(u_1)$.
Consequently,
\begin{align*}
\rho(G')-\rho(G)
&\geq \mathbf{x}^{\mathrm{T}}\big(A(G')-A(G)\big)\mathbf{x}
= 2x_{v_1}\Big(\sum\limits_{v\in N_{G}(u_1)}x_v-\sum\limits_{v\in N_{G}(v_1)}x_v\Big)\\
&=2x_{v_1}\rho(G)(x_{u_1}-x_{v_1})
\geq 0.
\end{align*}
If $\rho(G')=\rho(G)$,
then $\mathbf{x}$ is also a positive eigenvector of $G'$.
Observe that $N_G(u_1)\neq N_G(v_1)$.
Combined with the fact that $x_{u_1}\geq x_{v_1}$,
there must be a vertex $w\in N_G(u_1)\setminus N_G(v_1)$.
Then
\begin{align*}
\rho(G') x_{w}
=\sum_{v\in N_{G}(w)}x_v+x_{v_1}
>\sum_{v\in N_{G}(w)}x_v=\rho(G) x_{w},
\end{align*}
which contradicts $\rho(G')=\rho(G)$.
Thus, $\rho(G')>\rho(G)$.
By construction, $G'$ remains a blow-up of $F$,
which implies that $\chi(G')=\chi(F)\geq 4$ and $G'$ is triangle-free.
This contradicts the extremal property of $G$.
We conclude that there exists at most one subset $S'$ of $S$ with $|V_1(S')|\geq 2$.

For any $i\in \{1,2\}$, by Claim \ref{claim3.6} and $|\mathcal{A}_i|\leq 2^{20}$,
there exists a unique subset $S^*\subseteq S$ such that
 $|V_i(S^*)|\geq |V_i|-|\mathcal{A}_i|\geq (\frac{1}{2}-2\varepsilon^{\frac12})n.$
By the definition of $S$, we must have $S^*=\varnothing$.
Indeed, if not, then there would exist a vertex $u\in S^*$,
and we have $d_G(u)\geq |V_i(S^*)|\geq (\frac{1}{2}-2\varepsilon^{\frac12})n$,
which contradicts $d_G(u)\leq \frac{2n}{5}$ (since $u\in S$).
Consequently,
$|V_i(\varnothing)| \geq (\frac{1}{2}-2\varepsilon^{\frac12})n$ for each $i\in \{1,2\}$,
as required.
\end{proof}

\begin{claim}\label{claim3.7}
We have $e(S)\geq 1$ and $|S|=3$.
\end{claim}

\begin{proof}

We first prove that $e(S)\geq 1$.
Otherwise, $S$ is an independent set of $G$.
Then $S\cup V_1'\cup V_2'$ forms the color classes of $G$,
which implies that $\chi(G)\leq 3$, contradicting $\chi(G)\geq 4$.

Next, we prove that $|S|\leq 3$.
Suppose not, and we
select four vertices $u_1,u_2,u_3,u_4\in S$
and set $G'=G-\{u_1,u_2,u_3,u_4\}$.
Recall that $|\mathcal{A}_i|\leq 2^{20}$.
For each $i\in \{1,2\}$ and $j\in \{1,2,3,4\}$, we have $d_{V_i}(u_j)\leq |\mathcal{A}_i\setminus \{\varnothing\}|< 2^{20}.$
It follows that
$$d_G(u_j)\leq d_{V_1}(u_j)+d_{V_2}(u_j)+|S|\leq 2^{22}.$$
Since $G'$ is triangle-free, Tur\'{a}n's theorem gives $\rho^2(G')\leq \frac{(n-4)^2}{4}$.
Recursively applying Lemma \ref{Lem2.4},
we can obtain that
\begin{align*}
\rho^2(G)&\leq \rho^2(G-\{u_1\})+2\cdot 2^{22} \leq   \rho^2(G')+8\cdot 2^{22}\leq \frac{(n-4)^2}{4}+8\cdot 2^{22},
\end{align*}
which contradicts the bound $\rho(G)\geq \frac{n-3}{2}-\frac{13.2}{n}$ given by Claim \ref{Claim3.1}.
Thus, $|S|\leq 3$.

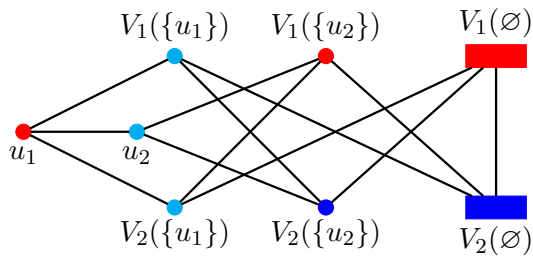
\begin{figure}[!h]
\centering
\begin{tikzpicture}[scale=0.5, x=1.00mm, y=1.00mm, inner xsep=0pt, inner ysep=0pt, outer xsep=0pt, outer ysep=0pt]
\definecolor{L}{rgb}{0,0,0}
\definecolor{F}{rgb}{0,0,0}

\node[rectangle,fill=red,draw=red,inner sep=0pt,minimum width=8mm, minimum height=3mm,label={[label distance=1mm]90:$V_1(\varnothing)$}] (rect1) at (105,20) {}; % 绘制一个长方形
\node[rectangle,fill=blue,draw=blue,inner sep=0pt,minimum width=8mm, minimum height=3mm,label={[label distance=1mm]270:$V_2(\varnothing)$}] (rect2) at (105,-20) {}; % 绘制一个长方形

\node[circle,fill=cyan,draw=cyan,inner sep=0pt,minimum size=2mm,label={[label distance=1mm]270:$u_2$}] (x) at (10,0) {};
\node[circle,fill=red,draw=red,inner sep=0pt,minimum size=2mm,label={[label distance=1mm]270:$u_1$}] (y) at (-20,0) {};

\node[circle,fill=red,draw=red,inner sep=0pt,minimum size=2mm,label={[label distance=1mm]90:$V_1(\{u_2\})$}] (A1) at (60,20) {};
\node[circle,fill=cyan,draw=cyan,inner sep=0pt,minimum size=2mm,label={[label distance=1mm]90:$V_1(\{u_1\})$}] (A2) at (20,20) {};

\node[circle,fill=blue,draw=blue,inner sep=0pt,minimum size=2mm,label={[label distance=0.5mm]270:$V_2(\{u_2\})$}] (B1) at (60,-20) {};
\node[circle,fill=cyan,draw=cyan,inner sep=0pt,minimum size=2mm,label={[label distance=0.5mm]270:$V_2(\{u_1\})$}] (B2) at (20,-20) {};

\definecolor{L}{rgb}{0,0,0}
\path[line width=0.3mm, draw=L] (x) -- (y);
\path[line width=0.3mm, draw=L] (x) -- (A1);
\path[line width=0.3mm, draw=L] (x) -- (B1);

\path[line width=0.3mm, draw=L] (y) -- (A2);
\path[line width=0.3mm, draw=L] (y) -- (B2);

\path[line width=0.3mm, draw=L] (B2) -- (rect1);

\path[line width=0.3mm, draw=L] (A2) -- (rect2);

\path[line width=0.3mm, draw=L] (rect1) -- (rect2);

\path[line width=0.3mm, draw=L] (A1) -- (rect2);

\path[line width=0.3mm, draw=L] (B1) -- (rect1);

\path[line width=0.3mm, draw=L] (A1) -- (B2);
\path[line width=0.3mm, draw=L] (A2) -- (B1);

\end{tikzpicture}
\caption{The graph $G$.}{\label{fig-1.2}}
\end{figure}

Since $G[S]$ is non-empty, we have $|S|\geq 2$.
Now assume $|S|=2$ and write $S=\{u_1,u_2\}$.
Since $G[S]$ is non-empty, it follows that $u_1u_2\in E(G)$.
Moreover, since $G$ is triangle-free, we have $V_i(S)=\varnothing$.
If $V_1(\{u_1\})=\varnothing$,
then $\{u_2\}\cup (V_1'\cup \{u_1\})\cup V_2'$ forms the color classes of $G$,
which implies $\chi(G)\leq 3$.
However, this contradicts $\chi(G)\geq 4$.
Thus, $V_1(\{u_1\})\neq \varnothing$.
By the same argument, we conclude that $V_1(\{u_2\}),V_2(\{u_1\}),V_2(\{u_2\})\neq \varnothing$.
Figure \ref{fig-1.2} illustrates that
 $$\big(\{u_1\}\cup V_1(\{u_2\}\cup V_1(\varnothing)\big)\cup \big(V_2(\{u_2\}\cup V_2(\varnothing)\big)
\cup \big(V_1(\{u_1\}\cup \{u_2\}\cup V_2(\{u_1\}\big)$$
forms the color classes of $G$,
yielding $\chi(G)\leq 3$.
This contradicts $\chi(G)\geq 4$.
Therefore, $|S|=3$, as required.
\end{proof}

In view of Claim \ref{claim3.7},
we may assume that $S=\{u_1,u_2,u_3\}$ and $u_1u_2\in E(G)$.

\begin{claim}\label{Claim3.8}
The subgraph $G[S]$ contains exactly one edge $u_1u_2$.
\end{claim}

\begin{proof}
Since \(\chi(G) \geq 4\) and $V_2'$ is an independent set, it follows that $\chi(G-V_2') \geq 3$.
Hence, $G[\{u_1,u_2,u_3\} \cup V_1']$ contains an odd cycle $C$.
Since $V_1'$ is an independent set, at most $\frac{|V(C)|-1}{2}$ vertices on $C$ lie in $V_1'$.
It follows that at least \(\frac{|V(C)|+1}{2}\) vertices belong to $\{u_1,u_2,u_3\}$.
Consequently, $|V(C)| \leq 5$.
Since $G$ is triangle-free, we have $|V(C)| = 5$ and $\{u_1,u_2,u_3\}$ are all on $C$.
We may assume that $C=u_1u_2v_{23}u_3v_{13}u_1$.
Then, $u_2u_3\notin E(G)$ (otherwise, $\{u_2,u_3,v_{23}\}$ induces a triangle, a contradiction).
Similarly, $u_1u_3\notin E(G)$.
We conclude that $G[S]$ contains exactly one edge $u_1u_2$.
\end{proof}

\begin{claim}\label{claim3.9D}
Let $G'=G-S$.
We have $\rho(G)\leq \rho(G')+\frac{0.1}{n}$.
\end{claim}

\begin{proof}
Since $x_{u^*}=\max\{x_v: v\in V(G)\}$,
we have $\rho(G)x_{u^*}=\sum_{v\in N_{G}(u^*)}x_v\leq d_G(u^*)x_{u^*}$,
which yields that $d_G(u^*)\geq \rho(G)\geq (\frac12-\varepsilon^{\frac12})n$.
Thus, $u^*\notin S$.
We may assume without loss of generality that $u^*\in V_1'$.
In light of Claim \ref{claim3.7}, we have $|S|=3$,
and hence $|\mathcal{A}_i|\leq 2^{|S|}=8$ for $i=1,2$.
For any vertex $u\in V_1'$, by Claim \ref{claim3.6},
we have $$|N_G(u^*)\setminus N_G(u)|\leq |S|+(|\mathcal{A}_2|-1)\leq 10.$$
Thus, $\rho(G)x_{u^*}-\rho(G)x_{u}\leq 10x_{u^*}$, which yields 
    $x_u\geq \frac{\rho(G)-10}{\rho(G)}x_{u^*}\geq \frac{1}{2}x_{u^*}.$
Hence,
 $$\sum_{u\in V(G)}x_u^2\geq \sum_{u\in V_1'}x_u^2\geq \frac{1}{4}x_{u^*}^2\cdot \frac{1}{3}n.$$
This, together with  $\sum_{u\in V(G)}x_u^2=1$, gives that $x_{u^*}^2\leq \frac{12}{n}$.
For $i\in \{1,2,3\}$, we have $d_G(u_i)\leq 1+|\mathcal{A}_1\setminus \{\varnothing\}|+|\mathcal{A}_2\setminus \{\varnothing\}|\leq 15$.
Then, $\rho(G)x_{u_i}=\sum_{v\in N_G(u_i)}x_v\leq 15x_{u^*}$,
and so $x_{u_i}\leq \frac{15x_{u^*}}{\rho(G)}$.
It follows that
\begin{align*}
\rho(G')-\rho(G)
&\geq \mathbf{x}^{\mathrm{T}}\big(A(G')-A(G)\big)\mathbf{x}
\geq -2\sum\limits_{i=1}^3\sum\limits_{v\in N_G(u_i)}x_{u_i}x_v   \\
&\geq -90\cdot \frac{15x_{u^*}}{\rho(G)}x_{u^*}\geq -\frac{0.1}{n},
\end{align*}
where the last inequality holds as $\rho(G)\geq \frac{n-3}{2}-\frac{12.2}{n}$, $x_{u^*}^2\leq \frac{12}{n}$, and $n$ is sufficiently large.
\end{proof}

From the construction of $C$, both $V_1(\{u_1,u_3\})$ and $V_1(\{u_2,u_3\})$ are non-empty.
Moreover, Claim \ref{claim3.6} implies that each of  $V_1(\{u_1,u_3\})$ and $V_1(\{u_2,u_3\})$ contains exactly one vertex.
Hence we may write $V_1(\{u_1,u_3\})=\{v_{13}\}$ and $V_1(\{u_2,u_3\})=\{v_{23}\}$.
Since $u_1u_2\in E(G)$ and $G$ is triangle-free, it follows that $V_1(\{u_1,u_2\})=V_1(\{u_1,u_2,u_3\})=\varnothing$.
The following two claims concern the sets $V_i(u_j)$ for $i=1,2$ and $j=1,2,3$.

\begin{claim}\label{Claim3.9}
We have $V_1(\{u_3\})=V_2(\{u_3\})=\varnothing$.
\end{claim}

\begin{proof}
By symmetry, it suffices to prove that $V_1(\{u_3\})=\varnothing$.
Suppose to the contrary,  then by Claim \ref{claim3.6},
 we may assume that $V_1(\{u_3\})=\{v_3\}$.
Let $G'$ be the graph obtained from $G$ by adding the edge $u_1v_3$.
Then, $\chi(G')\geq \chi(G)\geq 4$ and $\rho(G')>\rho(G)$.
By the choice of $G$, it follows that $G'$ contains a triangle containing the edge $u_3v_3$.
Let $w$ denote the third vertex of the triangle. 
Since $u_1u_3,u_2u_3\notin E(G)$, it follows that $w\notin \{u_1,u_2\}$.
Moreover, $w\notin V_1'$ because $V_1'$ is an independent set.
Thus, $w\in V_2'$, and $u_3w,v_3w\in E(G)$.
Since $G$ is triangle-free,
we see that $w\neq w_{13}$ (otherwise, $\{u_3,w_{13},v_3\}$ induces a triangle, a contradiction) and
$$v_{13}w_{13},v_{13}w_{23},v_{23}w_{13},v_{23}w_{23},
v_{13}w,w_{13}v_3,w_{23}v_3\notin E(G).$$
Let $G''=G-S$.
Then $G''$ can be obtained from $K_{|V_1'|,|V_2'|}$ by deleting at least 7 edges.
By Lemma \ref{Lem2.3}, we obtain
$$\rho(G'')\leq \rho(K_{|V_1'|,|V_2'|})-\frac{13.9}{n}\leq \sqrt{|V_1'|\cdot |V_2'|}-\frac{13.9}{n}
\leq \frac{n-3}{2}-\frac{13.9}{n}.$$
Combining this with Claim \ref{claim3.9D} yields $\rho(G)\leq \frac{n-3}{2}-\frac{13.8}{n}$,
which contradicts Claim \ref{Claim3.1}.
Thus, $V_1(\{u_3\})=\varnothing$.
Similarly, we also get $V_2(\{u_3\})=\varnothing$.
\end{proof}

\begin{claim}\label{claim3.10}
Exactly two of $V_1(\{u_1\})$, $V_2(\{u_1\})$, $V_1(\{u_2\})$, and $V_2(\{u_2\})$ are empty sets.
\end{claim}

\begin{proof}
We begin by proving that at least two of $V_1(\{u_1\})$, $V_2(\{u_1\})$, $V_1(\{u_2\})$, and $V_2(\{u_2\})$ must be empty.
Suppose for contradiction that at most one of them is empty.
Without loss of generality, assume that $V_1(\{u_2\})=\{v_2\}$, $V_2(\{u_1\})=\{w_1\}$, and $V_2(\{u_2\})=\{w_2\}$.
Since $G$ is triangle-free,
we see that
   $$v_{13}w_{13},v_{13}w_{23},v_{23}w_{13},v_{23}w_{23},v_2w_{23},v_2w_2,v_{13}w_1\notin E(G).$$
Let $G''=G-S$.
Then $G''$ can be obtained from $K_{|V_1'|,|V_2'|}$ by deleting at least 7 edges.
By Lemma \ref{Lem2.3}, we obtain
$$\rho(G'')\leq \rho(K_{|V_1'|,|V_2'|})-\frac{13.9}{n}\leq \sqrt{|V_1'|\cdot |V_2'|}-\frac{13.9}{n}
\leq \frac{n-3}{2}-\frac{13.9}{n}.$$
Combining this with Claim \ref{claim3.9D} yields $\rho(G)\leq \frac{n-3}{2}-\frac{13.8}{n}$,
which contradicts Claim \ref{Claim3.1}.
Thus, at least two of $V_1(\{u_1\})$, $V_2(\{u_1\})$, $V_1(\{u_2\})$, and $V_2(\{u_2\})$ are empty sets.

We then prove that exactly two of $V_1(\{u_1\})$, $V_2(\{u_1\})$, $V_1(\{u_2\})$, and $V_2(\{u_2\})$ 
are empty sets.
For the sake of contradiction, suppose that exactly three of them are non-empty
(the case where all of $V_1(\{u_1\})$, $V_2(\{u_1\})$, $V_1(\{u_2\})$, and $V_2(\{u_2\})$ are empty is similar and hence omitted here).
We may assume without loss of generality that $V_1(\{u_2\})$ is non-empty.
By Claim \ref{claim3.5},
it is not hard to verify that $G'$ is a blow-up of $F_1-\{v_1\}$ (see $F_1$ in Figure \ref{fig-1.1A}).
From Figure \ref{fig-1.1A}, we observe that $\chi(F_1-\{v_1\})\leq 3$.
Consequently, $\chi(G)=\chi(F_1-\{v_1\})\leq 3$, a contradiction.
\end{proof}

\begin{claim}\label{claim3.15}
There exists an integer $k$ with $|k|\leq 20$ such that $G\cong F_1(\frac{n-9-k}{2},\frac{n-9+k}{2}).$
\end{claim}

\begin{proof}
By Claim \ref{claim3.10},
exactly two of $V_1(\{u_1\})$, $V_2(\{u_1\})$, $V_1(\{u_2\})$, and $V_2(\{u_2\})$ are empty sets.
Without loss of generality, let $V_2(\{u_1\})=\varnothing$.
The proof proceeds in three cases.

\begin{figure}[!h]
\centering
% 第一个子图：占页面45%宽度，左对齐
\begin{minipage}{0.45\textwidth}
\centering
\begin{tikzpicture}[scale=0.45, x=1.00mm, y=1.00mm, inner xsep=0pt, inner ysep=0pt, outer xsep=0pt, outer ysep=0pt]
\definecolor{L}{rgb}{0,0,0}
\definecolor{F}{rgb}{0,0,0}

\node[circle,fill=red,draw=red,draw,inner sep=0pt,minimum size=2mm,label={[label distance=1mm]90:$x$}] (rect1) at (80,20) {};
\node[circle,fill=blue,draw=blue,inner sep=0pt,minimum size=2mm,label={[label distance=1mm]270:$y$}] (rect2) at (80,-20) {};

\node[circle,fill=red,draw=red,inner sep=0pt,minimum size=2mm,label={[label distance=1mm]270:$u_2$}] (x) at (10,0) {};
\node[circle,fill=blue,draw=blue,inner sep=0pt,minimum size=2mm,label={[label distance=1mm]270:$u_1$}] (y) at (-10,0) {};
\node[circle,fill=blue,draw=blue,inner sep=0pt,minimum size=2mm,label={[label distance=1mm]180:$u_3$}] (z) at (30,0) {};

\node[circle,fill=cyan,draw=cyan,inner sep=0pt,minimum size=2mm,label={[label distance=1mm]90:$v_{23}$}] (A1) at (40,20) {};
\node[circle,fill=cyan,draw=cyan,inner sep=0pt,minimum size=2mm,label={[label distance=1mm]90:$v_{13}$}] (A2) at (20,20) {};
\node[circle,fill=F,draw,inner sep=0pt,minimum size=2mm,label={[label distance=0.5mm]90:$v_2$}]
 (Ax) at (60,20) {};

\node[circle,fill=cyan,draw=cyan,inner sep=0pt,minimum size=2mm,label={[label distance=0.5mm]270:$w_{23}$}]
(B1) at (40,-20) {};
\node[circle,fill=cyan,draw=cyan,inner sep=0pt,minimum size=2mm,label={[label distance=0.5mm]270:$w_{13}$}]
(B2) at (20,-20) {};
\node[circle,fill=blue,draw=blue,inner sep=0pt,minimum size=2mm,label={[label distance=0.5mm]270:$w_2$}]
(Bx) at (60,-20) {};

\definecolor{L}{rgb}{0,0,0}
\path[line width=0.3mm, draw=L] (x) -- (y);
\path[line width=0.3mm, draw=L] (x) -- (A1);
\path[line width=0.3mm, draw=L] (x) -- (B1);
\path[line width=0.3mm, draw=L] (x) -- (Ax);
\path[line width=0.3mm, draw=L] (x) -- (Bx);

\path[line width=0.3mm, draw=L] (y) -- (A2);
\path[line width=0.3mm, draw=L] (y) -- (B2);

\path[line width=0.3mm, draw=L] (z) -- (A2);
\path[line width=0.3mm, draw=L] (z) -- (B2);
\path[line width=0.3mm, draw=L] (z) -- (B1);
\path[line width=0.3mm, draw=L] (A1) -- (z);

\path[line width=0.3mm, draw=L] (B2) -- (Ax);
\path[line width=0.3mm, draw=L] (B2) -- (rect1);

\path[line width=0.3mm, draw=L] (A2) -- (Bx);
\path[line width=0.3mm, draw=L] (A2) -- (rect2);

\path[line width=0.3mm, draw=L] (Ax) -- (rect2);
\path[line width=0.3mm, draw=L] (Bx) -- (rect1);

\path[line width=0.3mm, draw=L] (rect1) -- (rect2);

\path[line width=0.3mm, draw=L] (A1) -- (rect2);

\path[line width=0.3mm, draw=L] (B1) -- (rect1);

\draw(30,-40) node[anchor=base west]{\fontsize{12.23}{17.07}\selectfont $F_2$};

\end{tikzpicture}
\end{minipage}
\hspace{2mm}
% 第二个子图：占页面45%宽度，左对齐
\begin{minipage}{0.45\textwidth}
\centering
\begin{tikzpicture}[scale=0.45, x=1.00mm, y=1.00mm, inner xsep=0pt, inner ysep=0pt, outer xsep=0pt, outer ysep=0pt]
\definecolor{L}{rgb}{0,0,0}
\definecolor{F}{rgb}{0,0,0}

\node[circle,fill=cyan,draw=cyan,inner sep=0pt,minimum size=2mm,label={[label distance=1mm]90:$v_{13}$}] (u1) at (0,20) {};
\node[circle,fill=cyan,draw=cyan,inner sep=0pt,minimum size=2mm,label={[label distance=1mm]90:$v_{23}$}] (u2) at (20,20) {};
\node[circle,fill=red,draw=red,inner sep=0pt,minimum size=2mm,label={[label distance=1mm]90:$v_1$}] (u3) at (40,20) {};
\node[circle,fill=blue,draw=blue,inner sep=0pt,minimum size=2mm,label={[label distance=1mm]270:$w_2$}] (u4) at (40,-20) {};
\node[circle,fill=red,draw=red,inner sep=0pt,minimum size=2mm,label={[label distance=1mm]90:$x$}] (u5) at (60,20) {};

\node[circle,fill=blue,draw=blue,inner sep=0pt,minimum size=2mm,label={[label distance=1mm]180:$u_1$}] (u6) at (-30,0) {};
\node[circle,fill=red,draw=red,inner sep=0pt,minimum size=2mm,label={[label distance=1mm]270:$u_2$}] (u7) at (-10,0) {};
\node[circle,fill=blue,draw=blue,inner sep=0pt,minimum size=2mm,label={[label distance=1mm]180:$u_{3}$}] (u11) at (10,0) {};

\node[circle,fill=cyan,draw=cyan,inner sep=0pt,minimum size=2mm,label={[label distance=1mm]270:$w_{13}$}] (u8) at (0,-20) {};
\node[circle,fill=cyan,draw=cyan,inner sep=0pt,minimum size=2mm,label={[label distance=1mm]270:$w_{23}$}] (u9) at (20,-20) {};
\node[circle,fill=blue,draw=blue,inner sep=0pt,minimum size=2mm,label={[label distance=1mm]270:$y$}] (u10) at (60,-20) {};

\path[line width=0.3mm, draw=L] (u1) -- (u6);
\path[line width=0.3mm, draw=L] (u3) -- (u6);
\path[line width=0.3mm, draw=L] (u2) -- (u7);
\path[line width=0.3mm, draw=L] (u4) -- (u7);
\path[line width=0.3mm, draw=L] (u6) -- (u7);
\path[line width=0.3mm, draw=L] (u6) -- (u8);
\path[line width=0.3mm, draw=L] (u7) -- (u9);

\path[line width=0.3mm, draw=L] (u1) -- (u10);
\path[line width=0.3mm, draw=L] (u2) -- (u10);
\path[line width=0.3mm, draw=L] (u3) -- (u10);
\path[line width=0.3mm, draw=L] (u4) -- (u5);
\path[line width=0.3mm, draw=L] (u5) -- (u10);

\path[line width=0.3mm, draw=L] (u5) -- (u8);
\path[line width=0.3mm, draw=L] (u5) -- (u9);

\path[line width=0.3mm, draw=L] (u1) -- (u11);
\path[line width=0.3mm, draw=L] (u2) -- (u11);
\path[line width=0.3mm, draw=L] (u8) -- (u11);
\path[line width=0.3mm, draw=L] (u9) -- (u11);

\path[line width=0.3mm, draw=L] (u4) -- (u1);
\path[line width=0.3mm, draw=L] (u3) -- (u9);
\path[line width=0.3mm, draw=L] (u3) -- (u4);

\draw(30,-40) node[anchor=base west]{\fontsize{12.23}{17.07}\selectfont $F_3$};

\end{tikzpicture}
\end{minipage}

\caption{The graphs $F_2$ and $F_3$.}{\label{fig-1.3A}}

\end{figure}
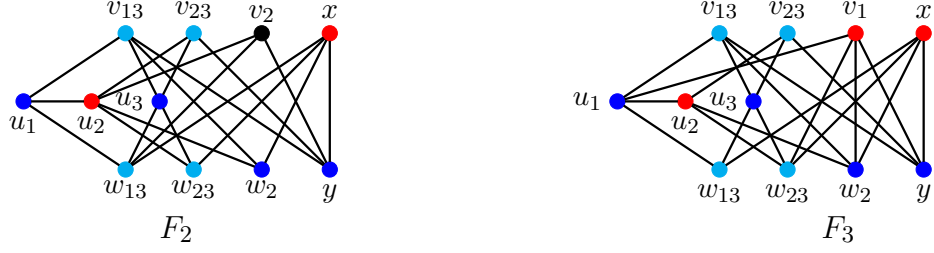

\noindent{\textbf{Case 1.} $V_1(\{u_1\})=\varnothing$.}

It is straightforward to verify that $G$ is a blow-up of $F_2$ (see $F_2$ in Figure \ref{fig-1.3A}).
Let $F_2(s,t)$ be the graph obtained from $F_2$ by replacing $x$ and $y$ with independent sets of sizes $s$ and $t$, respectively,
such that two vertices of $F_2(s,t)$ are adjacent if and only if their corresponding original vertices in $F_2$ are adjacent.
From the definition of $F_2(s,t)$,
we know that $G\cong F_2(\frac{n-9-k}{2},\frac{n-9+k}{2})$ for some $k\leq \frac{n-9}{2}$.
For short, set $G'=G-\{u_1,u_2,u_3\}$.
It is not hard to verify that $G'$ can be obtained from $K_{\frac{n-3-k}{2},\frac{n-3+k}{2}}$
by deleting 7 edges.
Setting $\alpha_1=0$ and $\alpha_2=7$ in Lemma \ref{Lem2.3}, we obtain
$| \rho(G)-\rho(K_{\frac{n-3-k}{2},\frac{n-3+k}{2}})+\frac{14}{n}| \leq\frac{0.1}{n},$
which yields that
\begin{align*}
\rho(G)\leq \rho(K_{\frac{n-3-k}{2},\frac{n-3+k}{2}})-\frac{13.9}{n}\leq \frac{n-3}{2}-\frac{13.9}{n}.
\end{align*}
Combining this with Claim \ref{claim3.9D} yields $\rho(G)\leq \frac{n-3}{2}-\frac{13.8}{n}$,
which contradicts Claim \ref{Claim3.1}.

\noindent{\textbf{Case 2.} $V_1(\{u_2\})=\varnothing$.}

It is straightforward to verify that $G'$ is a blow-up of $F_3$.
From Figure \ref{fig-1.3A}, we see that $\chi(F_3)\leq 3$.
Hence, $\chi(G)=\chi(F_3)\leq 3$, a contradiction.

\noindent{\textbf{Case 3.} $V_2(\{u_2\})=\varnothing$.}

One can easily verify that $G'$ is a blow-up of $F_1$.
Specifically, $G\cong F_1(\frac{n-9-k}{2},\frac{n-9+k}{2})$ for some $k\leq \frac{n-9}{2}$.
For short, set $G''=G-\{u_1,u_2,u_3\}$.
It is evident that $G''$ is a subgraph of $K_{\frac{n-1-k}{2},\frac{n-5+k}{2}}$,
which implies $$\rho^2(G'')\leq \rho^2(K_{\frac{n-1-k}{2},\frac{n-5+k}{2}})=\frac{(n-3)^2-(k-2)^2}{4}.$$
As seen in Figure \ref{fig-1.1A}, $d_G(u_i)=4$ for each $i\in \{1,2,3\}$.
By repeatedly applying Lemma \ref{Lem2.4},
we obtain that
\begin{align*}
\rho^2(G)&\leq \rho^2(G-\{u_1\})+2d_{G}(u_1)
     \leq   \rho^2(G')+2d_{G}(u_1)+2d_{G}(u_2)+2d_{G}(u_3)\\
     &\leq \rho^2(G'')+24\leq \frac{(n-3)^2}{4}-\frac{(k-2)^2}{4}+24.
\end{align*}
On the other hand, by Claim \ref{Claim3.1}, $\rho^2(G)\geq \big(\frac{n-3}{2}-\frac{13.2}{n}\big)^2\geq \frac{(n-3)^2}{4}-14$.
Combining the above two inequalities yields that $|k|\leq 20$,
as required.
\end{proof}

Let $s=k-2$. Since $|k|\leq 20$, we get $|s|\leq 22$ and
$G\cong F_1(\frac{n-9-k}{2},\frac{n-9+k}{2})=F_1(\frac{n-11-s}{2},\frac{n-7+s}{2})$.
Then by Lemma \ref{lem4.2} and the extremality of $G$, it follows that $G=F_1(\lfloor\frac{n-11}{2}\rfloor,\lceil\frac{n-7}{2}\rceil)$.
This completes the proof of Theorem \ref{thm1.1}.
\end{proof}

\end{document}